\documentclass[10pt]{amsart}

\usepackage{amssymb,amsmath,amsthm,mathrsfs,amsfonts,graphicx,epsf}
\usepackage{color}
\usepackage{xcolor}
\usepackage{hyperref}            

 \usepackage{esint}
 \usepackage{fullpage}
\usepackage{enumerate}

\newcommand{\Reg}{\mathsf{R}}

\newcommand{\Sing}{\mathsf{S}}

\newcommand{\bD}{{\mathcal D}}

\newcommand{\mup}{\hat \mu^p}
\newcommand{\muq}{\hat \mu^q}

\newcommand{\ps}{\langle\cdot,\cdot\rangle}
\renewcommand{\span}{\mathrm{span}}

\newtheorem{theorem}{Theorem}[section]
\newtheorem{corollary}[theorem]{Corollary}
\newtheorem{lemma}[theorem]{Lemma}
\newtheorem{proposition}[theorem]{Proposition}

\theoremstyle{definition}
\newtheorem{definition}[theorem]{Definition}

\theoremstyle{remark}
\newtheorem{remark}[theorem]{Remark}
\newtheorem{example}[theorem]{Example}
\newcommand{\bt}{\begin{theorem}}
\newcommand{\et}{\end{theorem}}
\newcommand{\bl}{\begin{lemma}}
\newcommand{\el}{\end{lemma}}
\newcommand{\bp}{\begin{proposition}}
\newcommand{\ep}{\end{proposition}}
\newcommand{\bc}{\begin{corollary}}
\newcommand{\ec}{\end{corollary}}
\newcommand{\bdeff}{\begin{definition}}
\newcommand{\edeff}{\end{definition}}
\newcommand{\brem}{\begin{remark}}
\newcommand{\erem}{\end{remark}}
\renewcommand{\r}[1]{(\ref{#1})}

\def\diam{\mathop{\mathrm{diam}}}
\def\ss{\mathcal{S}}
\newcommand{\hh}{{\mathcal H}}

\newcommand{\bi}{\begin{itemize}}
\newcommand{\iii}{\item}
\newcommand{\ei}{\end{itemize}}
\newcommand{\bd}{\begin{description}}
\newcommand{\ed}{\end{description}}

\newcommand{\bqn}{\begin{eqnarray}}
\newcommand{\eqn}{\end{eqnarray}}
\newcommand{\eqnn}{\nonumber\end{eqnarray}}

\newcommand{\ba}[1]{\begin{array}{#1}}
\newcommand{\ea}{\end{array}}

\newcommand{\R}{\mathbb{R}}

\newcommand{\N}{\mathbb{N}}

\newcommand{\g}{\gamma}

\newcommand{\eps}{\epsilon}

\newcommand{\VecM}{\mathrm{Vec}(M)}

\title{On measures in sub-Riemannian geometry}

\author[R.~Ghezzi]{Roberta Ghezzi}
\address{Institut de Math\'ematiques de Bourgogne UBFC, 9 Avenue Alain Savary BP47870 21078 Dijon Cedex France}
\email{roberta.ghezzi@u-bourgogne.fr}

\author[F.~Jean]{Fr\'{e}d\'{e}ric Jean}
\address{Unit\'{e} de Math\'{e}matiques Appliqu\'{e}es, ENSTA ParisTech, Universit\'{e} Paris-Saclay,
F-91120 Palaiseau, France}
\email{frederic.jean@ensta-paristech.fr}

\thanks{This work was partially supported by iCODE (Institute for Control and Decision),
research project of the IDEX Paris--Saclay, by the ANR project SRGI ``Sub-
Riemannian Geometry and Interactions'', contract number ANR-15-CE40-0018, and by a public grant as part of the
Investissement d'avenir project, reference ANR-11-LABX-0056-LMH, LabEx
LMH, in a joint call with Programme Gaspard Monge en Optimisation et
Recherche Op\'{e}rationnelle.}

\begin{document}

\begin{abstract}
In \cite{gjha} we give a detailed analysis of spherical Hausdorff measures on sub-Riemannian manifolds in a general framework, that is, without the assumption of equiregularity. The present paper is devised as a complement of this analysis, with both new results and open questions.
\end{abstract}

\maketitle

\tableofcontents

\section{Introduction}

\subsection{Objectives and structure of the paper}
In \cite{gjha} we give a detailed analysis of spherical Hausdorff measures on sub-Riemannian manifolds in a general framework, that is, without the assumption of equiregularity. The present paper is devised as a complement of this analysis, with both new results and open questions. The first aim is to extend the study to other kinds of intrinsic measures on sub-Riemannian manifolds, namely Popp's measure and general (i.e., non spherical) Hausdorff measures. The second is to explore some consequences of \cite{gjha} on metric measure spaces based on sub-Riemannian manifolds.

We choose to give first in this introduction a readable and synthetic presentation in the form of an informal discussion. We then provide all general definitions in Section~\ref{se:notations}. We study in Section~\ref{se:mgh} measured Gromov--Hausdorff  convergence in sub-Riemannian geometry and we state open questions on the behaviour of general  Hausdorff measures. Finally
Section~\ref{se:popp} contains the results on Popp's measure in the presence of singular points.

\subsection{Setting}
Let $(M, \bD, g)$ be a sub-Riemannian manifold: $M$ is a
smooth manifold, $\bD$ a Lie-bracket generating distribution on $M$  and $g$ a   Riemannian metric on $\bD$ (note that our framework  will permit us to consider rank-varying distributions as well). As in Riemannian geometry,  one defines the length of
  absolutely continuous paths which are
almost everywhere tangent to $\bD$   by integrating the $g$-norm of
their tangent vectors.  Then, the sub-Riemannian
distance $d$ is defined as the infimum of length of paths between two given points. The Lie-bracket generating assumption implies that, for every point $p\in M$ there exists $r(p)\in\N$ such that
\begin{equation}\label{flagintro}
\{0\}=\bD^0_p\subset\bD_p^1\subset\dots\subset\bD_p^{r(p)}=T_pM,
\end{equation}
where $\bD^i_p=\{X(p)\mid X\in\bD^i\}$ and $\bD^i\subset\VecM$ is the submodule defined recursively  by
  $\bD^{1}=\bD$, $\bD^{i+1}=\bD^{i}+[\bD,\bD^{i}]$. A point $p$ is \emph{regular}  if for every integer $i$ the dimension $\dim\bD^i_q$ is locally constant near $p$. Otherwise, $p$ is said to be \emph{singular}. Finally, for $p \in M$ we set
  \begin{equation*}
  Q(p)=\sum_{i=1}^{r(p)}i(\dim\bD^i_p-\dim\bD^{i-1}_p).
  \end{equation*}

 We first  discuss properties of Hausdorff measures near regular points; then  we give some constructions and estimates of Popp's measures in the presence of singular points (i.e.\ in non equiregular manifolds).\medskip

\subsection{On Hausdorff measures near regular points}
\label{se:haus_intro}
Let us recall some results on the spherical Hausdorff measures in sub-Riemannian geometry (see Theorem~3.1 and Proposition~5.1(iv) in \cite{gjha}).
Let $U$ be a connected component of the set $\Reg$ of regular points (which is an open subset of $M$).
\begin{enumerate}[(i)]
  \item The Hausdorff dimension of $U$ is $\dim_H U=Q$, where $Q$ is the constant value of $Q(p)$ for $p \in U$.
  \item The spherical Hausdorff measure $\ss^{Q}$ is a Radon measure on $U$.
\end{enumerate}
Assume moreover $M$  to be oriented and consider a smooth volume $\mu$  on $M$, i.e., a measure defined on open sets by $\mu(A)=\int_A \omega$, where $\omega\in\Lambda^n M$ is a positively oriented non degenerate $n$-form.
\footnote{Orientation is needed here to have a globally defined non degenerate $n$-form. However without this hypothesis all our results could be stated locally with a locally defined non degenerate $n$-form.}

\begin{enumerate}[(i)]
\setcounter{enumi}{2}
  \item $\ss^{Q}$ and $\mu$ are mutually absolutely continuous.
  \item The Radon--Nikodym derivative  $\frac{d\ss^{Q}}{d\mu}(p)$, $p \in U$, coincides with  the density $\lim_{\eps\to 0}\frac{\ss^{Q}( B(p,\eps))}{\mu(B(p,\eps))}$, whose value is
\begin{equation}\label{densintro}
\lim_{\eps \to 0}\frac{\ss^{Q}(B(p,\eps))}{\mu(B(p,\eps))}=\frac{2^Q}{\mup(\widehat B_p)},
\end{equation}
where $B(p,\eps)$ is the sub-Riemannian ball centered at $p$ of radius $\eps$, $\widehat B_p$ is the unit ball of the nilpotent approximation $(T_pM,\widehat d_p)$ at $p$, and $\mup$ is a measure on $T_pM$ obtained through a blow-up procedure of $\mu$ at $p$.
 \item The function $p \mapsto \frac{d\ss^{Q}}{d\mu}(p)$ is continuous on $U$.
\end{enumerate}

For the usual Hausdorff measure $\hh^{Q}$, since $\hh^Q \leq\ss^Q \leq 2^Q \hh^Q$, properties (ii)-(iv) hold: $\hh^{Q}$ is a Radon measure on $U$, $\hh^{Q}$ and $\mu$ are mutually absolutely continuous, and
\begin{equation*}
\frac{d\hh^{Q}}{d\mu}(p)=\lim_{\eps\to 0}\frac{\hh^{Q}( B(p,\eps))}{\mu(B(p,\eps))}.
\end{equation*}
However we have no formula such as \eqref{densintro} for this density. We would like to discuss here the different interpretations and implications of \eqref{densintro} and its potential extensions to $\hh^Q$.

\paragraph{\emph{Measured Gromov--Hausdorff convergence.}}
Note that $2^Q=\ss_{\widehat d_p}^Q (\widehat B_p)$, where $\ss_{\widehat d_p}^Q$ is the $Q$-dimensional spherical Hausdorff measure of the metric space $(T_pM,\widehat d_p)$. Thus \eqref{densintro} writes as
\begin{equation*}
\lim_{\eps \to 0}\frac{\ss^{Q}(B(p,\eps))}{\mu(B(p,\eps))}=\frac{\ss_{\widehat d_p}^Q (\widehat B_p)}{\mup(\widehat B_p)}.
\end{equation*}
Knowing that $(T_pM,\widehat d_p)$ is the metric tangent cone to $(M,d)$ at $p$, the above formula suggests a kind of convergence of the measures $\ss^{Q}$ and $\mu$ to $\ss_{\widehat d_p}^Q$ and $\mup$ respectively through a blow-up procedure. The appropriate notion is the one of measured Gromov--Hausdorff convergence of pointed metric measure spaces (see Definition~\ref{mgh}), and we actually prove in Section~\ref{se:mghS} the following results.

\begin{theorem}
\begin{itemize}
  \item For every $p\in M$, $(M,\frac1\eps d,\frac{1}{\eps^{Q(p)}}\mu, p)$ converges to $(T_pM,\widehat d_p,\mup,0)$ in the measured Gromov--Hausdorff sense as $\eps\to 0$.
  \item For every regular point $p\in M$, $(M,\frac1\eps d,\ss^Q_{d/\eps}, p)$ converges to $(T_pM,\widehat d_p,\ss_{\widehat d_p}^Q,0)$ in the measured Gromov--Hausdorff sense as $\eps\to 0$, where $Q=Q(p)$ and $\ss^Q_{d/\eps}$ denotes the $Q$-dimensional spherical Hausdorff measure on $M$ associated with the distance $\frac1\eps d$.
\end{itemize}
\end{theorem}

The key point in the proof of the second result is the fact that the Radon--Nikodym derivative  $\frac{d\ss^{Q}}{d\mu}(q)$ depends continuously on $q$ near a regular point $p$. This raises a first question:\medskip

\noindent \textbf{Question 1.} \ \emph{Is the function $q \mapsto \frac{d\hh^{Q}}{d\mu}(q)$ continuous near a regular point?}

\noindent If it is the case, then, when $p$ is a regular point, $(M,\frac1\eps d,\hh^Q_{d/\eps}, p)$ admits a limit for the measured Gromov--Hausdorff convergence as $\eps\to 0$.

\paragraph{\emph{Isodiametric constant.}}
In point (iv) above we obtain the Radon--Nikodym derivative  $\frac{d\ss^{Q}}{d\mu}(p)$ through the usual density $\lim_{\eps\to 0}\frac{\ss^{Q}( B(p,\eps))}{\mu(B(p,\eps))}$ (this is possible since both measures are Radon). We could rather use Federer densities, which yields the formula (see \cite[2.10.17(2),2.10.18(1)]{federer} or \cite[Theorem 11]{Magnani2015}):
\begin{equation*}
\frac{d\mu}{d\ss^{Q}}(p)= \lim_{\eps \to 0} \sup \left\{ \frac{\mu(B)}{(\mathrm{diam} B)^{Q}} \ : \ B \hbox{ closed ball}, \  p \in B, \  0< \mathrm{diam} B < \eps  \right\}.
\end{equation*}
Note that for closed balls $B=\overline{B}(q,r)$, we have
\begin{equation*}
\frac{\mu(\overline{B}(q,r))}{(\mathrm{diam} \overline{B}(q,r))^{Q}} = \frac{\mu(\overline{B}(q,r))}{(2r)^{Q}} \to \frac{\muq(\widehat B_q)}{2^Q} \hbox{ as }r \to 0,
\end{equation*}
and, by \cite[Proposition~5.1(iv)]{gjha}, in the neighbourhood of a regular point the convergence above is uniform w.r.t.\ $q$ and the limit is continuous. We recover in this way \eqref{densintro} (actually the proof of the latter formula in \cite{gjha} already used these properties of uniform convergence and continuity).

We can apply the same strategy to the usual Hausdorff measure. Using Federer densities (see \cite[2.10.17(2),2.10.18(1)]{federer},  or \cite[Theorem 10]{Magnani2015}), we obtain
\begin{equation*}
\frac{d\mu}{d\hh^{Q}}(p)= \lim_{\eps \to 0} \sup \left\{ \frac{\mu(S)}{(\mathrm{diam} S)^{Q}} \ : \ S \hbox{ closed subset of }M, \  p \in S, \  0< \mathrm{diam} S < \eps  \right\}.
\end{equation*}
This formula is  interesting when applied to $\mu=\ss^Q$. In that case it takes the form
\begin{equation*}
\frac{d\ss^Q}{d\hh^{Q}}(p)= \lim_{\eps \to 0} {\mathscr I}(p,\eps), \ \hbox{with } {\mathscr I}(p,\eps)=\sup \left\{ \frac{\ss^Q(S)}{(\mathrm{diam} S)^{Q}} \ : \ S \hbox{ closed subset of }M, \  p \in S, \  0< \mathrm{diam} S < \eps  \right\}.
\end{equation*}
When the sub-Riemannian manifold has a structure of a Carnot group, ${\mathscr I}:={\mathscr I}(p,\eps)$ is independent of $p$ and $\eps$ and it is called the \emph{isodiametric constant} of the group. In particular, at a regular point $p$, the nilpotent approximation $(T_pM,\widehat d_p)$ has a structure of a Carnot group. We will denote by ${\mathscr I}_p$ its isodiametric constant and it turns out that ${\mathscr I}_p$ is the multiplicative factor relating $Q$-dimensional Hausdorff and spherical Hausdorff measures in the Carnot group, i.e., there holds
$\ss^Q_{\widehat d_p}={\mathscr I}_p \hh^Q_{\widehat d_p}$ (recall that $Q=Q(p)$).
This raises a second question:\medskip

\noindent \textbf{Question 2.} \ \emph{Let $p$ be a regular point. Does ${\mathscr I}(p,\eps)$ converge to ${\mathscr I}_p$ as $\eps \to 0$?}\medskip

\noindent When the answer is positive, we obtain a formula similar to \eqref{densintro} for the $Q$-dimensional Hausdorff measure, since its Radon--Nikodym derivative w.r.t.\ a smooth volume $\mu$ at a regular point $p$ is given by
\begin{equation*}
\frac{d\hh^{Q}}{d\mu}(p)= \frac{2^Q}{{\mathscr I}_p \, \mup(\widehat B_p)}.
\end{equation*}
In that case, Question 1 above amounts to the following problem:\medskip

\noindent \textbf{Question 3.} \ \emph{Is the isodiametric constant ${\mathscr I}_q$ of the nilpotent approximation at $q$ continuous w.r.t.\ $q$  near a regular point?}\medskip

\noindent A positive answer to both questions 2 and 3 would imply the convergence in the measured Gromov--Hausdorff sense of $(M,\frac1\eps d,\hh^Q_{d/\eps}, p)$  to $(T_pM,\widehat d_p,\hh_{\widehat d_p}^Q,0)$ when $p$ is a regular point. All these issues are discussed in detail in Section~\ref{se:mgh}.

\subsection{On Popp's measure in the presence of singular points}

Popp's measure is a smooth volume on sub-Riemannian manifolds defined near regular points that is intrinsically associated to the sub-Riemannian structure. It was introduced first in \cite{montgomery} and then used in \cite{ABGR} to define an intrinsic Laplacian in the sub-Riemannian setting. An explicit formula for Popp's measure in terms of adapted frames is given in \cite{br}.

Popp's measure $\mathscr{P}$ is the smooth volume associated with a volume form which is built by a suitable choice of inner product structure on the graded vector space
  \begin{equation*}
\mathfrak{gr}_p(\bD)=\bD^1_p\oplus\bD^2_p/\bD^1_p\oplus\dots\oplus\bD^{r(p)}_p/\bD^{r(p)-1}_p.
\end{equation*}
This graded vector space depends smoothly on $p$  only near regular points, hence Popp's measure is defined only on the open set $\Reg$ of regular points. This raises two questions.\medskip

\noindent \textbf{Question 1.} How does $\mathscr{P}$ behave near singular points?\medskip

\noindent \textbf{Question 2.} Is it possible to extend the notion of Popp's measure to the whole manifold, including the set of singular points?\medskip

As for the first question, we prove that Popp's measure behaves as the spherical Hausdorff measure. More precisely, assume $(M, \bD, g)$ is an oriented sub-Riemannian manifold and fix a smooth volume $\mu$ on $M$. Consider an open connected subset $U$ of $\Reg$, denote by $\mathscr{P}^U$ the Popp measure on $U$ and by $Q$ the Hausdorff dimension of $U$.

\begin{theorem}\label{plop}
For any compact subset $K \subset M$ there exists a constant
$C>0$ such that, for every $q \in U \cap K$,
\begin{equation*}
C^{-1} \frac{d\ss^Q}{d\mu}(q)  \leq \frac{d\mathscr{P}^{U}}{d\mu}(q) \leq C \frac{d\ss^Q}{d\mu}(q).
\end{equation*}
\end{theorem}

As a consequence, the detailed analysis of the behaviour of  $\frac{d\ss^Q}{d\mu}(q)$ provided in \cite{gjha} also  applies to $\frac{d\mathscr{P}^{U}}{d\mu}(q)$. We stress that the compact set $K$ in Theorem~\ref{plop} may not be contained in $U$.  We have in particular the following properties if the boundary $\partial U$ contains a singular point $p$:
\begin{itemize}
  \item $\frac{d\mathscr{P}^{U}}{d\mu}(q) \to \infty$ as $q \to p$;
  \item $\frac{d\mathscr{P}^{U}}{d\mu}(q)$ may be not $\mu$-integrable near $p$; in other terms, the measure of balls $\mathscr{P}^{U} (B(q,r) \cap U)$ may tend to $\infty$ as $q \to p$ (the radius $r$ being fixed), or equivalently, $\mathscr{P}^{U}$, considered as a measure on $M$ by setting $\mathscr{P}^{U}=0$ on $M \setminus U$, may not be Radon measure on $M$.
\end{itemize}

To answer the second question, we recall the notion of \emph{equisingular submanifolds}, see \cite{gjha}. These are submanifolds $N$ of $M$ on which the restricted graded vector space
\begin{equation*}
 \mathfrak{gr}_p^N(\bD)= \bigoplus_{i=1}^{r(p)}(\bD^{i}_p\cap T_pN)/(\bD^{i-1}_p\cap T_pN)
 \end{equation*}
 depends smoothly on $p \in N$. A construction similar to the one of $\mathscr{P}^{U}$ gives rise to a smooth volume $\mathscr{P}^{N}$ on $N$. Note that the Hausdorff dimension of an equisingular submanifold can be computed algebraically as
 \begin{equation*}
\dim_H N = Q_N = \sum_{i=1}^{r(q)} i \left(\dim (\bD_q^i\cap T_qN) - \dim (\bD_q^{i-1}\cap T_qN) \right) \quad \hbox{for any } q \in N,
 \end{equation*}
see \cite[Theorem 5.3]{gjha}.

Assume now that $M$ is \emph{stratified by equisingular submanifolds}, that is, $M$ is a countable union $\cup_i N_i$ of disjointed equisingular submanifolds $N_i$. This assumption is not a strong one since it is satisfied by any analytic sub-Riemannian manifolds and by generic $C^\infty$ sub-Riemannian structures.
Note that the regular set $\Reg$ contains the union of all the open strata $N_i$ and that the singular set $\Sing$ is of $\mu$-measure zero. Under this assumption we can construct two kinds of Popp's measures on $M$:

\begin{itemize}
  \item we set
  \begin{equation*}
  \mathscr{P}_1 = \sum_{i \in {\mathcal J}}\mathscr{P}^{N_i},
  \end{equation*}
  where ${\mathcal J}=\{i \mid \dim_H N_i=\dim_H M\}$ (as before we consider $\mathscr{P}^{N}$ as a measure on $M$ by setting $\mathscr{P}^{N}=0$ on $M \setminus N$), and we obtain a measure that is absolutely continuous w.r.t.\ the Hausdorff measures $\hh^{\dim_H M}$ and $\ss^{\dim_H M}$;
  \item or we set
\begin{equation*}
\mathscr{P}_2 = \mathscr{P}^{\Reg}=\sum_{i \in {\mathcal O}}\mathscr{P}^{N_i},
\end{equation*}
where ${\mathcal O}=\{i \mid \dim {N_i}=\dim M\}$, and we obtain a measure that is absolutely continuous w.r.t.\ any smooth volume.
\end{itemize}


\section{Notations}
\label{se:notations}

\paragraph{\emph{Hausdorff measures.}}
Let us first recall some basic facts on Hausdorff measures. Let $(M,d)$ be
a metric space. We denote by $\diam S$ the diameter of a set $S
\subset M$. 
Let $\alpha \geq 0$ be a real number. For every set $E \subset M$, the
\emph{$\alpha$-dimensional Hausdorff measure} $\hh^\alpha$ of $E$  is
defined as  $\hh^\alpha(E)
= \lim_{\eps \to 0^+} \hh^\alpha_\eps(E)$, where
\begin{equation*}
\hh^\alpha_\eps(E) = \inf \left\{ \sum_{i=1}^\infty  \left(\diam S_i\right)^\alpha
\, : \, E \subset \bigcup_{i=1}^\infty S_i, \ S_i \hbox{ nonempty set} ,
\ \diam S_i \leq \eps \right\},
\end{equation*}
and the \emph{$\alpha$-dimensional spherical Hausdorff measure} is
defined as $\ss^\alpha(E)
= \lim_{\eps \to 0^+} \ss^\alpha_\eps(E)$, where
\begin{equation*}
\ss^\alpha_\eps(E) = \inf \left\{ \sum_{i=1}^\infty  \left(\diam
S_i\right)^\alpha \, : \, E \subset \bigcup_{i=1}^\infty S_i, \ S_i \hbox{ is
  a ball}, \ \diam
S_i \leq \eps  \right\}.
\end{equation*}
For every set $E\subset M$, the non-negative number
\begin{equation*}
D=\sup\{\alpha\geq0\mid \hh^\alpha(E)=\infty\}=\inf\{\alpha\geq0\mid \hh^\alpha(E)=0\}
\end{equation*}
is called the {\it Hausdorff dimension of $E$} and it is denoted as $\dim_{H}E$. Notice that $\hh^D(E)$
may be $0$, $>0$, or $\infty$.
By construction, for every subset $S \subset M$,
\begin{equation}\label{SHS}
\hh^\alpha(S)\leq\ss^\alpha(S)\leq 2^\alpha\hh^\alpha(S),
\end{equation}
hence the Hausdorff dimension can be defined equivalently using spherical measures.

Given a metric space $(M,d)$ we denote by $B^d(p,\eps)$ the open ball $\{q\in M\mid d(p,q)<\eps\}$.


\paragraph{\emph{Sub-Riemannian manifolds.}}
In the literature a sub-Riemannian manifold
  is usually a triplet  $(M,\bD,g)$, where $M$ is a smooth (i.e., ${\mathcal C}^\infty$) manifold,    $\bD$
is a
subbundle of $TM$ of rank $m<\dim
M$ and $g$ is a Riemannian metric  on $\bD$.
Using $g$, the length of  horizontal curves, i.e.,
absolutely continuous curves which are almost everywhere tangent to
$\bD$, is well-defined.
 When $\bD$ is Lie bracket generating, the map $d:M\times M\to\R$  defined as the infimum of length of horizontal
curves  between two given points is a continuous distance
(Rashevsky-Chow Theorem), and it is called  sub-Riemannian
  distance.

In the sequel, we are going to deal with  sub-Riemannian manifolds with singularities. Thus we find it more natural to work in a larger setting, where the map $q\mapsto \bD_q$ itself may have singularities. This leads us to the following generalized definition, see \cite{ABB,bellaiche}.

\begin{definition}\label{SR} A {\it sub-Riemannian} structure on a manifold $M$ is a triplet $({\bf U},\ps,f)$ where $({\bf U},\ps)$ is a Euclidean vector bundle over $M$ (i.e., a vector bundle $\pi_{\bf U}:{\bf U}\to M$ equipped with a smoothly-varying scalar product $q\mapsto \langle\cdot,\cdot\rangle_q$ on the fibre ${\bf U}_q$) and $f$ is a morphism of vector bundles $f:{\bf U}\to TM$, i.e. a smooth map linear on fibers and such that, for every $u\in {\bf U}$, $\pi(f(u))=\pi_{\bf U}(u)$, where $\pi:TM\to M$ is the usual projection.
\end{definition}

Let $({\bf U},\ps,f)$ be a sub-Riemannian structure on $M$.  We define the submodule $\bD\subset \VecM$ as
\begin{equation}\label{module}
\bD=\{f\circ\sigma\mid \sigma \textrm{ smooth section of }{\bf U}\},
\end{equation}
 and  for $q \in M$ we set $\bD_q=\{X(q)\mid X\in\bD\}\subset T_qM$.  Clearly $\bD_q=f({\bf U}_q)$.
The  length of a tangent vector $v\in \bD_q$ is defined
as
\begin{equation}\label{qf}
g_q(v):=
\inf\{\langle u,u\rangle_q\mid f(u)=v, u\in {\bf U}_q\}.
\end{equation}
An absolutely continuous  curve $\gamma:[a,b]\to M$ is  {\it horizontal } if $\dot\gamma(t)\in \bD_{\g(t)}$ for almost every $t$.
If $\bD$ is Lie bracket generating, that is
\begin{equation}\label{bg}
\forall\, q\in M\quad \textrm{Lie}_q\bD=T_qM,
\end{equation}
then  the map $d:M\times M\to\R$  defined as the infimum of length of horizontal
curves  between two given points is a continuous distance as in the classic case. In this paper, all sub-Riemannian manifolds are assumed to satisfy the Lie bracket generating condition \r{bg}.

Let $({\bf U}, \ps, f)$ be a sub-Riemannian structure on a manifold $M$, and $\bD$, $g$ the corresponding module and quadratic form as defined in \r{module} and \r{qf}. In analogy with the constant rank case and to simplify notations, in the sequel we will refer to the sub-Riemannian manifold as the triplet $(M, \bD, g)$.

Given $i\geq 1$, define recursively  the
  submodule $\bD^i\subset \VecM$ by
  \begin{equation*}
\bD^{1}=\bD,\quad \bD^{i+1}=\bD^{i}+[\bD,\bD^{i}].
\end{equation*}
Fix $p\in M$ and set $\bD^i_p=\{X(p)\mid X\in\bD^i\}$.
The Lie-bracket generating assumption implies that
there exists an integer $r(p)$
such that
 \begin{equation}\label{flagd}
\{0\}= \bD_p^0 \subset\bD_p^1\subset\dots\subset\bD_p^{r(p)}=T_p M.
\end{equation}
Set $n_i(p)=\dim\bD^i_p$ and
\begin{equation}
\label{defq}
Q(p)=\sum_{i=1}^{r(p)}i(n_{i}(p)-n_{i-1}(p)).
\end{equation}
  To write $Q(p)$ in a different way, we define the {\it weights} of the flag \r{flagd} at $p$ as the integers $w_1(p),\dots, w_{n}(p)$ such that $w_i(p)=s$
if $\dim\bD^{s-1}_p < i \leq \dim\bD^{s}_p$. Then $Q(p)= \sum_{i=1}^n w_i(p)$.
We say that a point $p$ is {\it regular} if, for every $i$, $n_i(q)$ is constant  as $q$ varies in a neighborhood of $p$. Otherwise, the point is said to be {\it singular}. A sub-Riemannian manifold with no singular point is said to be \emph{equiregular}.\medskip

We end by introducing a short notation for balls. When it is clear from the context, we will omit the upscript $d$ for balls $B(q,r)$ in a manifold $M$ endowed with a sub-Riemannian distance. Given $p\in M$ we denote by $\widehat B_p$ the unit ball $B^{\widehat d_p}(0,1)$ in the nilpotent approximation at $p$.

 We refer the reader to \cite{ABB,bellaiche,montgomery} for a primer in sub-Riemannian geometry.

\section{Measured Gromov--Hausdorff convergence in sub-Riemannian geometry}
\label{se:mgh}
In this section we prove some facts concerning measured Gromov--Hausdorff convergence of sub-Riemannian manifolds.

\subsection{Measured Gromov--Hausdorff convergence}
Recall the notion of Gromov--Hausdorff convergence of pointed metric spaces. Given a metric space $(X,d)$ and two subsets $A,B\subset X$, the Hausdorff distance between $A$ and $B$ is
\begin{equation*}
d_H(A,B)=\inf\{\eps>0\mid A\subset I_\eps(B), B\subset I_\eps(A)\},
\end{equation*}
 where $I_\eps(E)=\{x\in X\mid \mathrm{dist}(x,E)<\eps\}$.
\begin{definition}
Let $(X_k,d_k)$, $k\in\N$ be  metric spaces and $x_k\in X_k$. We say that the pointed metric spaces $(X_k,d_k,x_k)$ \emph{converge to $(X_\infty,d_\infty,x_\infty)$ in the   Gromov--Hausdorff sense}, provided for any $R>0$, the quantity
\begin{eqnarray*}
\inf\left\{ \right.\eps>0\mid& \exists\, (Z, d) \textrm{ compact metric space, } \exists \, i_k:\overline{B}^{d_k}(x_k, R)\to Z, \,\exists \, i_\infty:\overline{B}^{d_\infty}(x_\infty,R)\to Z, \\
 &\left. \textrm{isometric embeddings such that } d_H(i_k(\overline{B}^{d_k}(x_k,R)),i_\infty(\overline{B}^{d_\infty}(x_\infty,R)))<\eps\right\}
\end{eqnarray*}
converges to zero as $k$ goes to $\infty$. Any Gromov--Hausdorff limit is unique up to isometry.
\end{definition}

A particular case of Gromov--Hausdorff convergence is the one where the sequence of spaces is constructed by a blow-up procedure of the distance around a given point.
Take a metric space $(X,d)$ and a point $x\in X$. We say that the pointed metric space $(X_\infty, d_\infty, x_\infty)$ is a {\it metric tangent cone} to $(X,d)$ at $x$ if
$(X, \frac{1}{\eps} d, x)$ converges to $(X_\infty, d_\infty, x_\infty)$ in the Gromov--Hausdorff sense as $\eps \to 0$.

Given a sub-Riemannian manifold $(M,{\mathcal D},g)$ and $p\in M$,  a metric tangent cone to $(M,d)$ at $p$ exists and it is equal (up to an isometry) to $(T_pM, \hat d_p, 0)$, where $\hat d_p$ is the sub-Riemannian distance associated with the nilpotent approximation at $p$
(see \cite{bellaiche} for the proof and the definition of nilpotent approximation).

When a measure is given on a metric space, a natural notion of convergence is that  of measured Gromov--Hausdorff convergence.

\begin{definition}[see, for instance, \cite{GMS}]\label{mgh}
Let $(X_k,d_k)$, $k\in\N$, be complete and separable metric spaces, $m_k$ be Borel measures on $X_k$ which are finite on bounded sets, and $x_k\in\mathrm{supp}(m_k)$. We say that the pointed metric measure spaces $(X_k,d_k,m_k,x_k)$ \emph{converge to $(X_\infty,d_\infty,m_\infty,x_\infty)$ in the  measured Gromov--Hausdorff sense}, provided that for any $\delta, R>0$ there exists $N(\delta,R)$ such that for all $k\geq N(\delta,R)$ there exists a Borel map $f_k^{R,\delta}:B^{d_k}(x_k,R)\to X_\infty$ such that
\begin{itemize}
\item[a)] $f_k^{R,\delta}(x_k)=x_\infty$,
\item[b)] $\sup_{x,y\in B^{d_k}(x_k,R)}|d_k(x,y)-d_\infty(f_k^{R,\delta}(x),f_k^{R,\delta}(y))|\leq \delta$,
\item[c)] the $\delta$-neighborhood of $f_k^{R,\delta}(B^{d_k}(x_k,R))$ contains $B^{d_\infty}(x_\infty,{R-\delta})$,
\item[d)] $(f_k^{R,\delta})_\sharp(m_k\llcorner_{B^{d_k}(x_k,R)})$ weakly converges\footnote{We introduce the following notation. Given two metric spaces $(X,d_X)$, $(Y,d_Y)$, a measurable map $T:X\to Y$ (where we consider Borel $\sigma$-algebras on $X$ and $Y$), and a non negative Radon measure $\mu$ on $X$, we define the push-forward
measure  $T_\#\mu$ on $Y$ by
\begin{equation*}
T_\# \mu(B)=\mu(T^{-1}(B)).
\end{equation*}
 This definition provides the change of variable formula
\begin{equation}\label{cov}
\int_Y\varphi \, dT_\# \mu=\int_X\varphi\circ T\, d\mu,
\end{equation}
for every bounded or nonnegative measurable map $\varphi:Y\to \R$.} to $m_{\infty}\llcorner_{B^{d_\infty}(x_\infty,R)}$ as $k\to\infty$ for almost every $R>0$.
\end{itemize}
\end{definition}
It is not hard to see that  conditions a), b), c) are equivalent to requiring that the sequence of pointed metric spaces $(X_k,d_k,x_k)$ converges to $(X_\infty,d_\infty, x_\infty)$ in the Gromov--Hausdorff sense. The only condition involving measures is the last one.

One would expect that, when $(X_k,d_k,x_k)$ converges to some $(X_\infty,d_\infty,x_\infty)$ in the Gromov--Hausdorff sense and $m_k$ is some measure whose construction is intrinsically associated to $d_k$ (e.g.\ $m_k$ a Hausdorff measure), then measured Gromov--Hausdorff convergence (with limit measure the one corresponding to the same construction in the limit space $(X_\infty,d_\infty)$) should be quite natural to prove. It turns out that this is not the case and in general one needs additional conditions. For instance, when $(X_k,d_k)$ are Alexandrov spaces having Hausdorff dimension $\alpha$ and converging to some $(X_\infty,d_\infty, x_\infty)$ in the Gromov--Hausdorff sense, a sufficient condition for the measured Gromov--Hausdorff convergence of $(X_k,d_k, {\mathcal H}^\alpha_{d_k}, x_k)$ to $(X_\infty,d_\infty, {\mathcal H}^\alpha_{d_\infty}, x_\infty)$
is that all metric spaces $(X_k,d_k)$ satisfy the same curvature bound (see \cite[Theorem~10.10.10]{bb}).

In the sequel we are going to show two results concerning measured Gromov--Hausdorff convergence in sub-Riemannian manifolds.

\subsection{Convergence results for smooth volumes and spherical Hausdorff measures}
\label{se:mghS}
We study first the case of smooth volumes in sub-Riemannian manifolds.

Let $(M,\bD,g)$ be a sub-Riemannian manifold. At a regular point $p\in M$ the metric tangent cone $(T_pM,\widehat d_p)$ to $(M,d)$ has a structure of a Carnot group.  Assume $M$ is oriented and let $\mu$ be a smooth volume on $M$ associated with a volume form $\omega$. If $p$ is a regular point, then
$\omega$   induces canonically a left-invariant volume form $\hat\omega_p$ on $T_pM$. We denote by $\mup$ the smooth volume on $T_pM$ defined by $\hat\omega_p$.

\begin{proposition}\label{smooth}
Let $(M,\bD,g)$ be an  oriented sub-Riemannian manifold and let $\mu$ be any smooth volume on $M$.  Then, for every regular point $p\in M$, $(M,\frac1\eps d,\frac{1}{\eps^{Q(p)}}\mu, p)$ converges to $(T_pM,\widehat d_p,\mup,0)$ in the measured Gromov--Hausdorff sense, as $\eps\to 0$.
\end{proposition}
\begin{proof}
Denote by $d_\eps$ the distance $d/\eps$ and let $R>0$. Note that the metric space $(B^{d_\eps}(p,R),d_\eps)$, i.e., the set $\{q\in M\mid \frac1\eps d(q,p)<R\}$ endowed with the distance $d_\eps$, coincides with  $(B^d(p, R\eps),d_\eps)$.
We already know that $(\overline{B}^{d}(p,R\eps),d_\eps, p)$ converges to $(\overline{B}^{\widehat{d}_p}(0,R),\widehat d_p,0)$ in the Gromov--Hausdorff sense.
To cast this convergence in the language of Definition~\ref{mgh}, we fix some privileged coordinates $(x_1,\dots,x_n)$ centered at $p$ and we define the function\footnote{With this definition of $f_\eps$, the image $f_\eps(B^{d}(p,{R\eps}))$ may not be entirely contained in  $B^{\widehat d_p}(0,R)$. To be more precise, one should take as $f_\eps(x)$ the projection of the point $\delta_{1/\eps}x$ in $B^{\widehat d_p}(0,R)$. For the sake of readability and taking into account \r{bbella},
we prefer to avoid introducing the projection.}
$f_\eps: B^{d}(p,R\eps)\to B^{\widehat d_p}(0,R)$ by $f_\eps(x)=\delta_{1/\eps}x$, where  $\delta_\lambda:\R^n\to\R^n$ is the nonhomogeneous dilation $\delta_\lambda(x_1,\dots,x_n)=(\lambda^{w_1(p)}x_1,\dots, \lambda^{w_n(p)}x_n)$. Then condition a) is satisfied by construction and condition b) is a consequence of the convergence
\begin{equation*}
\lim_{\eps\to 0}\sup_{x,x'\in \overline{B}^{d}_{R\eps}(p)}\left|\frac{d(x,x')}{\eps}-\widehat d_p(f_{\eps}(x),f_\eps(x'))\right|=0,
\end{equation*}
see for instance  \cite[Theorem~7.32]{bellaiche}. Condition c) follows by the fact that, as $\eps\to 0$,
\begin{equation}\label{bbella}
B^{\widehat d_p}(0,{R\eps(1-o(1))})\subset B^{d}(p, R\eps) \subset B^{\widehat d_p}(0,{R\eps(1+o(1))}),
\end{equation}
see \cite[Corollary~7.33]{bellaiche}.

To prove condition d) we must show that for every continuous function $h:T_pM\to \R$ there holds
\begin{equation*}
\lim_{\eps \to 0}\int_{B^{\widehat d_p}(0,R)}h\,d\mu_\eps =\int_{B^{\widehat d_p}(0,R)}h\, d\mup,
\end{equation*}
where
\begin{equation*}
\mu_\eps=(f_\eps)_\sharp\left(\frac{1}{\eps^{Q}}\mu\llcorner_{B^d(p, R\eps)}\right),
\end{equation*}
and we
set $Q:=Q(p)$. By the change of variable formula \r{cov},
\begin{equation*}
\int_{B^{\widehat d_p}(0,R)}h(x)\,d\mu_\eps(x) = \frac{1}{\eps^{Q}}\int_{B^d(p, R\eps)}h(\delta_{1/\eps}x)\,d\mu(x).
\end{equation*}
Thanks to \r{bbella},
\begin{equation*}
\lim_{\eps\to 0} \frac{1}{\eps^{Q}}\int_{B^d(p, R\eps)}h(\delta_{1/\eps}x)\,d\mu(x)=\lim_{\eps\to 0}\frac{1}{\eps^{Q}}\int_{B^{\widehat d_p}(0,{R\eps(1+o(1))})}h(\delta_{1/\eps}x)\,d\mu(x),
\end{equation*}
whenever one among the two limit exists.

Note that in the coordinates $x$ we can express $d\mu(x)$ as $\omega(\partial_{x_1},\dots ,\partial_{x_n})|_x d{\mathcal L}$, where $d{\mathcal L}=dx_1\dots d x_n$, and $d\mup (x)$ as $\omega(\partial_{x_1},\dots , \partial_{x_n})|_0 d{\mathcal L}$.
Now, since the Lebesgue measure ${\mathcal L}$ in privileged coordinates is homogeneous of order ${Q}$ with respect to dilations $\delta_\lambda$, applying the change of variable $y=\delta_{1/\eps}x$ we obtain
\begin{equation*}
\frac{1}{\eps^{Q}}\int_{B^{\widehat d_p}(0,{R\eps(1+o(1))})}h(\delta_{1/\eps}x)\,d\mu(x) =\int_{B^{\widehat d_p}(0,{R(1+o(1))})}h(y) \omega(\partial_{y_1},\dots, \partial_{y_n})|_{\delta_\eps y}\, d{\mathcal L}(y).
\end{equation*}
Since the function $h$ is continuous and the function $\omega(\partial_{y_1},\dots, \partial_{y_n})$ is smooth,
we deduce
\begin{equation}\label{uti}
\lim_{\eps\to 0}\int_{B^{\widehat d_p}(0,{R(1+o(1))})}h(y) \omega(\partial_{y_1},\dots, \partial_{y_n})|_{\delta_\eps y}\, d{\mathcal L}(y)= \int_{B^{\widehat d_p}(0,R)}h(y) \,d\mup(y),
\end{equation}
which concludes the proof.
\end{proof}
\begin{remark}
In the proof above, we use inclusions \r{bbella} to infer that the required convergence can be proved on the nilpotent ball, which is a homogeneous set. Then, the main ingredient to deduce condition d) is the fact that the Radon--Nikodym derivative of $\mu$ with respect to the Lebesgue measure in coordinates is continuous at $0$.
\end{remark}

The assumption of regularity of $p$ in Proposition~\ref{smooth} can be dropped. Indeed, without this assumption it is not possible in general to define $\hat\omega_p$ but it is possible to define the measure $\mup$ on $T_pM$: choose local coordinates $x$ near $p$, which identify $T_pM$ to $\R^n$, and set, for a Borel set $A \subset T_pM$,
\begin{equation}
\label{eq:defmup}
\mup(A) = \int_A \omega(\partial_{x_1},\dots , \partial_{x_n})|_0 dx_1 \dots d x_n.
\end{equation}
And this expression of $\mup$ is the only one that we need in the preceding proof (it is used in \eqref{uti}).
Note that the coordinates $x$ need not be adapted.
\begin{corollary}
The statement of Proposition~\ref{smooth}  holds at a singular point $p$.
\end{corollary}

\begin{remark}
With the definition \eqref{eq:defmup} of $\mup$, we have the following expansion, already known at regular points (see \cite[Corollary 28]{abbh}),
\begin{equation*}
\mu(B(p,\eps)) = \eps^{Q(p)} \mup (\widehat B_p) + o(\eps^{Q(p)}).
\end{equation*}
\end{remark}

We now consider measured Gromov--Hausdorff convergence of spherical Hausdorff measures.

\begin{proposition}
\label{sfer}
Let $(M,\bD,g)$ be a sub-Riemannian manifold and $p\in M$ be a regular point. Set $Q:=Q(p)$. Denote by $\ss^Q_{d/\eps}$ the $Q$-dimensional spherical Hausdorff measure on $M$ associated with the distance $\frac1\eps d$ and by $\ss_{\widehat d_p}^Q$ the $Q$-dimensional spherical Hausdorff measure on $T_pM$ associated with $\widehat d_p$. Then $(M,\frac1\eps d,\ss^Q_{d/\eps}, p)$ converges to $(T_pM,\widehat d_p,\ss_{\widehat d_p}^Q,0)$ in the measured Gromov--Hausdorff sense, as $\eps\to 0$.
\end{proposition}
The idea of the proof is the same as the one used to deduce Proposition~\ref{smooth} and it is based on the fact that, near a regular point, $\ss^Q$ is absolutely continuous with respect to $\mu$ and its Radon--Nikodym derivative (which is computed  explicitly in \cite{abbh}) is a continuous function.

\begin{proof}
Let $R>0$. Reasoning as in the proof of Proposition~\ref{smooth} it suffices to prove condition d), namely that, for every continuous function $h:T_pM\to\R$,
\begin{equation*}
\lim_{\eps \to 0}\int_{B^{\widehat d_p}(0,R)}h\,d\sigma_\eps =\int_{B^{\widehat d_p}(0,R)}h\, d\ss_{\widehat d_p}^Q,
\end{equation*}
where
\begin{equation*}
\sigma_\eps=(f_\eps)_\sharp\left(\ss^Q_{d/\eps}\llcorner_{B^d(p, R\eps)}\right),
\end{equation*}
and $f_\eps$ is defined as in the proof of Proposition~\ref{smooth}  in a system of privileged coordinates $x$ centered at $p$.

First of all, notice that by construction of Hausdorff measures, we have the identity $\ss^Q_{d/\eps}=\frac{1}{\eps^Q}\ss^Q$, where $\ss^Q$ is the $Q$-dimensional Hausdorff measure in $M$ associated with $d$.
Hence, using the change of variable formula \r{cov},
\begin{equation*}
\int_{B^{\widehat d_p}(0,R)}h(x)\,d\sigma_\eps(x) = \int_{B^d(p, R\eps)}h(\delta_{1/\eps}x)\,d\ss_{d/\eps}^Q(x)=\frac{1}{\eps^Q}\int_{B^d(p, R\eps)}h(\delta_{1/\eps}x)\,d\ss^Q(x).
\end{equation*}
Thanks to \r{bbella},
\begin{equation*}
\lim_{\eps\to 0} \frac{1}{\eps^{Q}}\int_{B^d(p, R\eps)}h(\delta_{1/\eps}x)\,d\ss^Q(x)=\lim_{\eps\to 0}\frac{1}{\eps^{Q}}\int_{B^{\widehat d_p}(0,{R\eps(1+o(1))})}h(\delta_{1/\eps}x)\,d\ss^Q(x),
\end{equation*}
whenever one among the two limit exists.
Let $\mu$ be a smooth volume on $M$ (or on an orientable neighbourhood of $p$). Since $p$ is regular, in a neighbourhood of this point $\ss^Q$ is absolutely continuous with respect to $\mu$, the Radon--Nikodym derivative $\frac{d\ss^Q}{d\mu}$ is given explicitly by \eqref{densintro}, i.e.\
\begin{equation*}
\frac{d\ss^Q}{d\mu}(x)=\frac{2^Q}{\hat \mu^x(\widehat B_x)},
\end{equation*}
and it is a continuous function. 
Thus
\begin{eqnarray*}
\lim_{\eps\to 0}\frac{1}{\eps^{Q}}\int_{B^{\widehat d_p}(0,{R\eps(1+o(1))})}h(\delta_{1/\eps}x)\,d\ss^Q(x)&=&\lim_{\eps\to 0}\frac{1}{\eps^{Q}}\int_{B^{\widehat d_p}(0,{R\eps(1+o(1))})}h(\delta_{1/\eps}x)\frac{d\ss^Q}{d\mu}(x)\,d\mu(x)\\
&=&\lim_{\eps\to 0}\int_{B^{\widehat d_p}(0,{R(1+o(1))})}h(y)\frac{d\ss^Q}{d\mu}(\delta_\eps y)\omega(\partial_{y_1},\dots, \partial_{y_n})|_{\delta_\eps y}\, d{\mathcal L}(y).
\end{eqnarray*}
where we apply the change of variable formula \r{cov} and express  $\mu$ in terms of the Lebesgue measure  in the chosen system of privileged coordinates.
By  convergence \r{uti}, we deduce that
\begin{equation*}
\lim_{\eps\to 0}\frac{1}{\eps^{Q}}\int_{B^{\widehat d_p}(0,{R\eps(1+o(1))})}h(\delta_{1/\eps}x)\,d\ss^Q(x)=\int_{B^{\widehat d_p}(0,R)}h(y)\frac{d\ss^Q}{d\mu}(0)\,d\mup(y).
\end{equation*}
 %
%
Now, since $p$ is a regular point, $T_pM$ has a structure of a Carnot group and $\mup$ and  $\ss^Q_{\widehat d_p}$ are both Haar measures on this group. As a consequence they are proportional with a coefficient equal to
\begin{equation*}
\frac{d\ss^Q_{\widehat d_p}}{d\mup}(x)=\frac{2^Q}{\hat \mu^p(\widehat B_p)}=\frac{d\ss^Q}{d\mu}(0), \qquad \forall\, x\in T_pM,
\end{equation*}
(remind that in privileged coordinates $p$ corresponds to $x=0$).
 Finally,
 \begin{equation*}
 \int_{B^{\widehat d_p}(0,R)}h(y)\frac{d\ss^Q}{d\mu}(0)\,d\mup(y)=\int_{B^{\widehat d_p}(0,R)}h(y)\,d\ss_{\widehat d_p}^Q,
 \end{equation*}
which ends the proof.
 \end{proof}

\begin{remark}
The key property in the proof of weak convergence of spherical Hausdorff measure is that it is absolutely continuous with respect to a smooth volume and with continuous Radon--Nikodym derivative.
To our best understanding, the idea of using a smooth volume to handle  spherical Hausdorff measure is the only way to get convergence (without further assumptions).
\end{remark}


\subsection{Hausdorff measures}
In this section we analyse the problem of measured Gromov--Hausdorff convergence for Hausdorff measures.

We start with a simple fact.
\begin{proposition}\label{yui}
Let $(M,\bD,g)$ be an  oriented sub-Riemannian manifold and $\mu$ be a smooth volume on $M$.
Let $U \subset M$ be a connected component of the set of regular points and $Q$ be the constant value of $Q(p)$ for $p \in U$. Then $\mu$ and $\hh^Q$ are mutually absolutely continuous on $U$ and the Radon--Nikodym derivative of $\hh^Q$ with respect to $\mu$ can be computed as a density, i.e.,
\begin{equation*}
\frac{d\hh^Q}{d\mu}(p)=\lim_{r\to 0}\frac{\hh^Q(B(p, r))}{\mu(B(p,r))},\quad \forall\, p\in U.
\end{equation*}
\end{proposition}
\begin{proof}
Since $\ss^Q$ is Borel regular and finite on compact sets, so is $\hh^Q$ since $\hh^Q \leq\ss^Q \leq 2^Q \hh^Q$. Hence $\hh^Q$ is a Radon measure and we apply \cite[Theorem~4.7]{simon} to $X=M$, $\mu_1=\mu$ and $\mu_2=\hh^Q$ to obtain the conclusions.
\end{proof}

Assume $p\in M$ is a regular point. In the spirit of Proposition~\ref{sfer}, a natural question is to ask if $(M,d/\eps,\hh^Q_{d/\eps}, p)$ converges to $(T_pM,\widehat d_p,\hh_{\widehat d_p}^Q,0)$ in the measured Gromov--Hausdorff sense, as $\eps\to 0$.

As we recall in the proof of Proposition~\ref{smooth}, conditions a), b) and c) of Definition~\ref{mgh} (which do not involve measures) are verified choosing $f_\eps: B^{d}(p,{R\eps})\to B^{\widehat d_p}(0,R)$ as $f_\eps(x)=\delta_{1/\eps}x$, where $x=(x_1,\dots,x_n)$ is a system of privileged coordinates centered at $p$. Hence the question reduces to the weak-$\ast$ convergence of the measures
\begin{equation*}
\eta_\eps=(f_\eps)_\sharp\left(\hh^Q_{d/\eps}\llcorner_{B^d(p,{R\eps})}\right),
\end{equation*}
to the measure $\hh_{\widehat d_p}^Q\llcorner_{B^{\widehat d_p}(0,R)}$.

Let us proceed as in Proposition~\ref{sfer}. There holds
\begin{eqnarray*}
\lim_{\eps\to 0}\int_{B^{\widehat d_p}(0,R)}h\,d\eta_\eps&=&\lim_{\eps\to 0}\frac{1}{\eps^Q}\int_{B^{\widehat d_p}(0,{R\eps(1+o(1))})}h(\delta_{1/\eps}x) d\hh^Q(x)\\
&=&\lim_{\eps\to 0}\int_{B^{\widehat d_p}(0,{R(1+o(1))})}h(y) \frac{d\hh^Q}{d\mu}(\delta_\eps y)\omega(\partial_{y_1},\dots, \partial_{y_n})|_{\delta_\eps y}\, d{\mathcal L}(y),
\end{eqnarray*}
whenever one of the limits above exists. At this point, in order to continue one needs to assess the existence and value of the limit
\begin{equation}\label{iop}
\lim_{\eps\to 0}\frac{d\hh^Q}{d\mu}(\delta_\eps y).
\end{equation}
In particular, concluding as in Proposition~\ref{sfer} we have the following sufficient conditions for measured Gromov--Hausdorff convergence.
\begin{proposition}
\label{prop:haus}
Under the assumptions of Proposition~\ref{yui}, if
 the function $x\mapsto \frac{d\hh^Q}{d\mu}(x)$ is continuous, then  $\eta_\eps\llcorner_{B^{\widehat d_p}(0,R)}$ weak-$\ast$ converges to $ \frac{d\hh^Q}{d\mu}(0)\mup\llcorner_{B^{\widehat d_p}(0,R)}$. If moreover $\frac{d\hh^Q}{d\mu}(x)=\frac{\hh^Q_{\widehat d_x}(\widehat B_x)}{\widehat\mu (\widehat B_x)}$, then $(M, d/\eps, \hh^Q_{d/\eps},p)$ converges to $(T_pM, \widehat d_p,  \hh_{\widehat d_p}^Q, 0)$ in the measured Gromov--Hausdorff sense.
\end{proposition}

Recall that, for spherical Hausdorff measures, to prove continuity of $\frac{d\ss^Q}{d\mu}$, the idea is to compute the limit
\begin{equation}
\lim_{r\to 0}\frac{\ss^Q(B(x,r))}{\mu(B(x,r))},
\end{equation}
 which turns out to equal $\frac{\ss^Q_{\widehat d_x}(\widehat B_x)}{\hat \mu^x(\widehat B_x)}$ and show that the latter  is  continuous as a function of $x$.
The difficulty in estimating the quotient $\frac{\hh^Q(B(x,r))}{\mu(B(x,r))}$
is that one needs uniform estimates of $\hh^Q$ at any set of small diameter covering balls centered at points in a neighborhood of $p$. This is the main (and hard) problem when one passes from analysis of $\ss^Q$ to that of $\hh^Q$: coverings in the construction of $\ss^Q$ are made with balls (i.e. sets of points having bounded distance from a fixed one), whereas for $\hh^Q$ coverings are much more general and made with any set.

We do not know whether $\frac{d\hh^Q}{d\mu}(x)$ coincides with the ratio $\frac{\hh^Q_{\widehat d_x}(\widehat B_x)}{\hat\mu^x(\widehat B_x)}$. Neither do we know that the function  $x\mapsto \frac{\hh^Q_{\widehat d_x}(\widehat B_x)}{\hat\mu^x(\widehat B_x)}$ is continuous. Nevertheless, we can relate last continuity to the notion of isodiametric constants in Carnot groups.

Indeed, we can study the Radon--Nikodym derivative of $\hh^Q$ w.r.t.\ $\ss^Q$ rather than $\mu$ since \eqref{densintro} implies
\begin{equation*}
\frac{d\hh^{Q}}{d\mu}(p)= \frac{2^Q}{\mup(\widehat B_p)} \frac{d\hh^{Q}}{d\ss^Q}(p).
\end{equation*}
Using Federer densities (see subsection~\ref{se:haus_intro}), the ratio
$\frac{d\ss^Q}{d\hh^{Q}}(p)$ is given by $\lim_{\eps \to 0} {\mathscr I}(p,\eps)$, where
\begin{equation*}{\mathscr I}(p,\eps)=\sup \left\{ \frac{\ss^Q(S)}{(\mathrm{diam} S)^{Q}} \ : \ S \hbox{ closed subset of }M, \  p \in S, \  0< \mathrm{diam} S < \eps  \right\}.
\end{equation*}

When the sub-Riemannian manifold itself has a structure of a Carnot group $\mathbb{G}$, the measures $\ss^Q$ and $\hh^Q$ on $\mathbb{G}$ are proportional as Haar measures and, by \cite[Proposition~2.3]{rigot}, they satisfy
\begin{equation}\label{dfgghj}
\ss^Q={\mathscr I} \hh^Q,
\end{equation}
where ${\mathscr I}$ is the so called {\it isodiametric constant} in the Carnot group defined by
\begin{equation}\label{ic}
{\mathscr I}=\sup\left\{\frac{\ss^Q(A)}{(\diam A)^Q} \mid 0<\diam A <\infty, A\subset \mathbb{G}\right\},
\end{equation}
and the diameter is computed with respect to the sub-Riemannian distance $d$ in $\mathbb{G}$. Hence, for $M=\mathbb{G}$,  ${\mathscr I}(p,\eps)={\mathscr I}$ and it is independent of $p$ and $\eps$.

At a regular point $p$, the nilpotent approximation $(T_pM,\widehat d_p)$ has a structure of a Carnot group $\mathbb{G}_p$. Note that the Carnot group may vary (both from an algebraic and metric viewpoint) as $p$ varies in $M$. Its isodiametric constant ${\mathscr I}_p$ satisfies
\begin{equation*}
{\mathscr I}_p= \frac{\ss_{\widehat d_p}^Q(\widehat B_p)}{\hh_{\widehat d_p}^Q(\widehat B_p)} \quad \hbox{and} \quad
\frac{\hh_{\widehat d_p}^Q(\widehat B_p)}{\mup(\widehat B_p)}=  \frac{2^Q}{{\mathscr I}_p\mup(\widehat B_p)}.
\end{equation*}

We then have the following characterization of the conditions of Proposition~\ref{prop:haus}.

\begin{proposition}
\label{ciprop}
Under the assumptions of Proposition~\ref{yui}, there holds $\frac{d\hh^Q}{d\mu}(p)=\frac{\hh^Q_{\widehat d_p}(\widehat B_p)}{\hat\mu^p (\widehat B_p)}$ if and only if ${\mathscr I}(p,\eps)$ converges to ${\mathscr I}_p$ as $\eps \to 0$. Moreover,
the function
\begin{equation*}
p\mapsto\frac{\hh_{\widehat d_p}^Q(\widehat B_p)}{\mup(\widehat B_p)}
\end{equation*}
is continuous if and only if the isodiametric constant $p\mapsto {\mathscr I}_p$ is continuous.
\end{proposition}

Unfortunately, very little is known about isodiametric constants in Carnot groups and even less is known about isodiametric sets, that are sets realizing the supremum in \r{ic}.

In this connection let us mention \cite{rigot}, where the author proves that in Carnot groups with sub-Riemannian distances  isodiametric sets  exist. Moreover, under some assumption on the sub-Riemannian distance called condition $({\mathscr C})$,  it is proved that balls are not isodiametric (see Theorem~3.6 in \cite{rigot}).
 By definition, condition  $({\mathscr C})$ is satisfied on a Carnot group $\mathbb{G}$ with a sub-Riemannian distance if for some (and hence any) point $x\in\mathbb{G}$ there exists $y\in\mathbb{G}\setminus\{x\}$ such that one can find a length-minimizing curve $\gamma:[a, b]\to\mathbb{G}$ from $x$ to $y$ which ceases to be length-minimizing after $y$.
The latter means that, if $T>0$ and $c: [a,b + T]\to\mathbb{G}$ is any horizontal curve such that $c|_{[a,b]}=\gamma$ then $c$ is not length-minimizing on $[a, b + T ]$.
Such condition is actually satisfied by Carnot group under a mild assumption on abnormal length-minimizers\footnote{See \cite{ABB} or \cite{Rifford2014} for the definition of abnormal and strictly abnormal minimizers.}.
 \begin{proposition}
 Let $\mathbb{G}$ be a Carnot group endowed with a sub-Riemannian distance. Assume that $\mathbb{G}$ is not the Euclidean space and does not admit strictly abnormal length-minimizers. Then condition $({\mathscr C})$ is satisfied. As a consequence, balls are not isodiametric or, equivalently, the isodiametric constant ${\mathscr I}$ in $\mathbb{G}$ satisfies ${\mathscr I}>1$.
 \end{proposition}
\begin{proof}
A Carnot group different from the Euclidean space admits at least one normal extremal trajectory $\gamma$ which is length-minimizing on an interval $[a,b]$ but not on $[a,b + T]$ for any $T>0$ (see for instance Theorem~12.17 in \cite{ABB}). Thus if $c: [a,b + T]\to\mathbb{G}$ is a length-minimizer coinciding with $\gamma$ on $[a,b]$, it must be an abnormal extremal trajectory on $[a,b + T]$, and actually a strictly abnormal one. Our assumption on $\mathbb{G}$ excludes the existence of a such a  $c$.
\end{proof}

We also cite  \cite{LRV}, where the analysis is specified to the Heisenberg group. In this framework, the authors show that isodiametric sets have Lipschitz boundary and they provide explicit solutions to a constrained isodiametric problem (obtained by restricting the supremum in \r{ic} to the class of spherically symmetric sets).

\section{Popp's measure in non equiregular manifolds}
\label{se:popp}

In this section we briefly recall the construction of Popp's measure and its main properties. We then extend the construction to equisingular submanifolds and provide two new notions of Popp's measure on sub-Riemannian manifolds having singular points.

\subsection{Popp's measure on the set of regular points}

Let $(M,\bD,g)$ be an oriented sub-Riemannian manifold and let $U\subset M$ be a connected component of the set of regular points. We follow the construction of Popp's measure given in \cite[Section 10.6]{montgomery} which is based on the two algebraic facts given below.
\begin{lemma}\label{alg}
\bi
\iii[a)] Let $E$ be an inner product space, and let $\pi:E\to V$ be a surjective linear map. Then $\pi$  induces an inner product on $V$ such that the associated norm is
\begin{equation*}
|v|_V=\min\{|e|_E\mid \pi(e)=v\}.
\end{equation*}
\iii[b)] Let $E$ be a vector space of dimension $n$ with a flag of linear subspaces
\begin{equation*}
\{0\}=F^0\subset F^1\subset F^2\subset \dots\subset F^r=E,
\end{equation*}
and let
\begin{equation*}
\mathfrak{gr}(F)=F^1\oplus F^2/F^1\oplus\dots\oplus F^{r}/F^{r-1},
\end{equation*}
 be the associated graded vector space. Then there is a canonical isomorphism $\theta:\Lambda^n E^* \to \Lambda^n\mathfrak{gr}(F)^*$.
\ei
\end{lemma}
  Let $p\in U$ and let $r(p)$ the first integer such that $\bD^{r(p)}=T_pM$.  Popp's measure is the smooth volume associated with a volume form which is built by a suitable choice of inner product structure on the graded vector space
  \begin{equation*}
\mathfrak{gr}_p(\bD)=\bD^1_p\oplus\bD^2_p/\bD^1_p\oplus\dots\oplus\bD^{r(p)}_p/\bD^{r(p)-1}_p.
\end{equation*}

Let us detail the construction. For every $i=1,\dots, r(p)$ consider the linear map
$\pi_i:\otimes^i\bD_p\to\bD_p^i/\bD^{i-1}_p$ given by
\begin{equation}
\label{eq:pii}
\pi_i(v_1\otimes\dots\otimes v_i)=[V_1,[V_2,\dots, [V_{i-1}, V_i]]](p) \mod \bD^{i-1}_p,
\end{equation}
where $V_1,\dots, V_i$ are smooth sections of $\bD$ satisfying $V_j(p)=v_j$. Such maps are well defined, i.e., they do not depend on the choice of  $V_1,\dots, V_i$ and are surjective. Apply   Lemma~\ref{alg} a) to $E=\otimes^i\bD_p$,   $V=\bD_p^i/\bD^{i-1}_p$ and $\pi=\pi_i$. Then, for every $i=1,\dots r(p)$, $\pi_i$ endows $\bD^i_p/\bD^{i-1}_p$  with an inner   product space  structure.
As a consequence, the graded vector space $\mathfrak{gr}_p(\bD)$ becomes an inner product space as well. Taking the wedge product of the elements of an orthonormal dual basis of $\mathfrak{gr}_p(\bD)$, we obtain an element $\mathscr{V}_p\in\Lambda^n\mathfrak{gr}_p(\bD)^*$, which is defined up to a sign. Finally consider the map $\theta_p: \Lambda^n T^*_pM \to \Lambda^n\mathfrak{gr}_p(\bD)^*$ obtained by applying  Lemma~\ref{alg} b) to $E=T_pM$, $F=\bD_p$. The element $\varpi_p=\theta_p^{-1}(\mathscr{V}_p)$ in $\Lambda^n T^*_pM $ is defined canonically up to a sign.

Moreover, near a regular point the graded vector space $\mathfrak{gr}_p(\bD)$ varies smoothly, and so does the inner product that we defined on $\mathfrak{gr}_p(\bD)$. As a consequence $p \mapsto \varpi_p$ defines locally a volume form, and then  by a standard gluing argument a volume form $\varpi$ on the whole (oriented) connected component $U$.
Popp's measure is the smooth volume associated with $\varpi$.
We denote it by $\mathscr{P}^U$ and, with an abuse of notation we keep the symbol $\mathscr{P}^U$ for the measure on $M$ which coincides with the former on $U$ and is $0$ elsewhere. Finally, we set  $\mathscr{P}^{\Reg}=\sum_{U}\mathscr{P}^U$, where $U$ varies among all connected components of $\Reg$. When $M=\Reg$ is equiregular, $\mathscr{P}^{\Reg}$ is known in literature as Popp's measure.

By construction, $\mathscr{P}^\Reg$ is a smooth volume on $\Reg$. Its Radon-Nikodym derivative with respect to the top-dimensional spherical Hausdorff measure\footnote{Recall that $q\mapsto Q(q)$ is locally constant on $\Reg$.} may be computed by \eqref{densintro} applied to $\mu=\mathscr{P}$,
\begin{equation*}
\frac{d\mathscr{P}^{\Reg}}{d\ss^{Q(q)}}(q)=\frac{\widehat{\mathscr{P}}_q(\widehat B_q)}{2^{Q(q)}},~~\forall q\in \Reg,
\end{equation*}
where $\widehat{\mathscr{P}}_q$ denote the Popp measure in the nilpotent approximation at $q$.

Using the explicit formula in \cite{br} for Popp's measure we can provide a weak equivalent of the Radon-Nikodym derivative of $\mathscr{P}^{\Reg}$ with respect to a smooth volume  on (any) sub-Riemannian manifold.
To do so, let us introduce first some notations.

We say that a family of vector fields $X_1,\dots, X_m$ is a  \emph{(global) generating family} for the sub-Riemannian structure $(\bD, g)$ if $\bD$ is globally generated by $X_1,\dots, X_m$ and
\begin{equation*}
g_q(v)=\inf\left\{\sum_{i=1}^m u_i^2~\Big\vert~  v=\sum_{i=1}^m u_iX_i(q)\right\} \qquad \hbox{for every } q \in M, \ v \in T_q M.
\end{equation*}
The existence of a generating family is a consequence of \cite[Corollary~3.26]{ABB}. Given such a family, for every integer $s>0$ the set of iterated Lie bracket $[X_{i_1},[X_{i_2},[
 \dots,[X_{i_{j-1}},X_{i_j}]\dots]]$ of length $j \leq s$ generates the module $\bD^s$. We say that $(Y_1,\dots,Y_n)$ is \emph{a frame of brackets adapted at $p$} if $(Y_1(p),\dots,Y_n(p))$ is a basis of $T_pM$, every $Y_i$ is an iterated Lie bracket of $X_1,\dots, X_m$, and the sum, over all $i$'s, of the lengths of the brackets $Y_i$'s equals $Q(p)$ (the latter condition guarantees that the basis is adapted to the flag \eqref{flagintro}).


Given a generating family of $(\bD, g)$ and a smooth volume $\mu$ on $M$ associated with a non-degenerate volume form $\omega$, we define the function $q\mapsto\nu(q)$ as
\begin{equation*}
{\nu(q)} = \max\{  \omega_q(Y_1(q),\dots,Y_n(q)) \mid (Y_1,\dots, Y_n) \textrm{ frame of brackets adapted at } q\}.
\end{equation*}
Note that $\nu(q) \geq 0$ since $\omega_q(Y_1(q),Y_2(q),\dots)=-\omega_q(Y_2(q),Y_1(q),\dots) $ and actually $\nu(q) >0$ since there exist frames of brackets at any point $q$. The function $\nu$ is continuous at regular points but not at singular ones. Indeed, if $(q_k)$ is a sequence of regular points converging to a singular point $\bar q$, then $\nu(q_k) \to 0$, whereas $\nu(\bar q)>0$.

\begin{proposition}\label{weq}
For any compact subset $K \subset M$ there exists a constant
$C>0$ such that, for every $q \in \Reg \cap K$,
\begin{equation}
\label{fweq}
\frac{1}{C \nu(q)} \leq \frac{d\mathscr{P}^{\Reg}}{d\mu}(q) \leq  \frac{C}{\nu(q)}.
\end{equation}
\end{proposition}

\begin{proof}
Recall first that if $p$ is a regular point, then $r(p) <n$. Hence only a finite number of families $(Y_1,\dots,Y_n)$ of iterated brackets may be frames of brackets adapted at regular points. Fix one of these families, $\mathrm{Y}=(Y_1,\dots,Y_n)$, and define $U_\mathrm{Y}$ as the set of points $p \in \Reg$ such that $\mathrm{Y}$ is a frame of brackets adapted at $p$ and $ \omega_p(Y_1(p),\dots,Y_n(p)) >\nu(p)/2$
(note that $\omega_p(Y_1(p),\dots,Y_n(p)) \leq \nu(p)$ by definition of $\nu$). This set $U_\mathrm{Y}$ is an open subset of $M$ (possibly empty) on which all dimensions $n_j(q)$, $j\in \N$, are constant. The union of all such $U_\mathrm{Y}$ covers $\Reg$.

Denoting by $dY_1, \dots, dY_n$ the  frame of $T^*M$ on $U_\mathrm{Y}$ dual to $\mathrm{Y}$, we write the $n$-form $\omega$ on $U_\mathrm{Y}$ as
\begin{equation*}
\omega = \omega(Y_1,\dots,Y_n) dY_1 \wedge \cdots \wedge dY_n.
\end{equation*}
Let us use now the characterization of Popp's measure given in \cite[Theorem 1]{br}: the volume form defining  Popp's measure, that we denote also $\mathscr{P}^{\Reg}$, is given on $U_\mathrm{Y}$  by
\begin{equation}
\label{eq:pomega}
\mathscr{P}^{\Reg} = \frac{1}{\sqrt{\prod_j \det B_j}} dY_1 \wedge \cdots \wedge dY_n = \frac{\omega}{ \omega(Y_1,\dots,Y_n)\sqrt{\prod_j \det B_j}},
\end{equation}
for some matrices $B_j$ defined from structure constants associated with $(Y_1,\dots,Y_n)$.
A careful analysis of these matrices $B_j$ (given by formula (4) in \cite{br}) shows that, for every $j \in \{1, \dots, r(p)\}$:
\begin{itemize}
  \item $B_j(q)=b_j^T(q) b_j(q)$ where $b_j(q)$ is a matrix with $n_j(p)-n_{j-1}(p)$ columns whose coefficients depend smoothly on $q$ in $U_\mathrm{Y}$;
  \item since we have chosen $Y_1,\dots,Y_n$ as iterated brackets, we can order the columns of the matrix $b_j$ so that it writes by block as
  $$
  b_j(q)=\begin{pmatrix}
      I \\
      D_j(q) \\
    \end{pmatrix}, \quad \hbox{where $I$ is the identity matrix of size $n_j(p)-n_{j-1}(p)$}.
  $$
\end{itemize}
As a consequence there exists a continuous nonnegative function $b$ on $U_\mathrm{Y}$ such that
\begin{equation*}
\frac{1}{1+ b(q)} \leq \frac{1}{\sqrt{\prod_j \det B_j}} \leq 1.
\end{equation*}
Using \eqref{eq:pomega} and the definition of $U_\mathrm{Y}$, we obtain
\begin{equation*}
 \frac{1}{(1+ b(q)) \nu(q)} \leq \frac{d\mathscr{P}^{\Reg}}{d\mu}(q) \leq  \frac{2}{\nu(q)} \quad \hbox{for every } q \in U_\mathrm{Y}.
\end{equation*}
The conclusion follows since $\Reg$ is the finite union of the subsets $U_\mathrm{Y}$, as the family $Y$ varies.
\end{proof}

\begin{remark}
Let us compare our estimates in formula~\r{fweq} with the first equality in formula \r{eq:pomega} (see \cite[Theorem~1]{br}.
Formula \r{eq:pomega} gives a quantitative expression for $\mathscr{P}^{\Reg}$. Even though Proposition~\ref{weq} only provides a weak equivalent of $\mathscr{P}^{\Reg}$, it allows to estimate directly whether $\mathscr{P}^{\Reg}(K\cap\Reg)$ is finite or not (see the discussion below Corollary~\ref{cococ} below for $\mu$-integrability of $d\mathscr{P}^\Reg/d\mu$). The finiteness of Popp's measure of such sets is somehow hidden in \r{eq:pomega} since it is encoded in the $n$-form $dY_1\wedge\dots\wedge dY_n$.
\end{remark}

Let $U \subset M$ be a connected component of the regular set. Recall that its Hausdorff dimension equals the constant value $Q$ of $Q(p)$ for $p\in U$. The following statement is a consequence of Proposition~\ref{weq} and of \cite[Proposition~3.12]{gjha}.
\begin{corollary}\label{cococ}
For any compact subset $K \subset M$ there exists a constant
$C>0$ such that, for every $q \in U \cap K$,
\begin{equation*}
\frac1C \leq \frac{d\mathscr{P}^U}{d\ss^{Q}}(q) \leq C.
\end{equation*}
\end{corollary}

Let us stress that the estimates in Proposition~\ref{weq} and Corollary~\ref{cococ} hold for any compact set in $M$ and, in particular, also for compact sets having singular points, i.e., not contained in $\Reg$. This is the main novelty with respect to known properties of $\mathscr{P}^\Reg$ for equiregular manifolds.

Indeed, assume $K$ contains singular points.
In this case, since $\nu(q)\rightarrow 0$ as $q$ tends to any singular point,
the function $d\mathscr{P}^\Reg/d\mu$ blows up and it is not $L_{loc}^\infty(\mu)$ in $K\cap\Reg$. Going further in the analysis of regularity of $d\mathscr{P}^\Reg/d\mu$, one may ask whether this function is $\mu$-integrable: it turns out that this may fail to be the case.
To see some sufficient conditions we refer the reader to sections 4 and 6 in \cite{gjha} where we study $\mu$-integrability of $\ss^{Q}\llcorner_\Reg$. Those conditions apply to $\mathscr{P}^\Reg$ as well, since by Corollary~\ref{cococ} $\mathscr{P}^U$ is commensurable with $\ss^{Q}\llcorner_U$, for any connected component $U\subset\Reg$ of Hausdorff dimension $Q$.

\begin{corollary}\label{ccor}
For any compact subset $K \subset M$ there exists a constant
$C>0$ such that
\begin{equation*}
C^{-1} \leq \mathscr{\widehat P}_q(\widehat B_q) \leq C, \textrm{ for every } q \in \Reg \cap K.
\end{equation*}
\end{corollary}
\begin{proof}
Recall that, if $q$ is a regular point, $\frac{d\mathscr{P}^\Reg}{d\mu}(q)=\frac{\mathscr{\widehat P}_q(\widehat B_q)}{\muq(\widehat B_q)}$.
The statement follows directly by Proposition~\ref{weq} and by \cite[Proposition~5.7]{gjha} which implies that there is $C>0$ such that, if $q\in K\cap\Reg$,
\begin{equation*}
C^{-1}\nu(q)\leq \muq(\widehat B_q)\leq C\nu(q).
\end{equation*}
\end{proof}
Again, the relevance of Corollary~\ref{ccor} is that $q\mapsto \mathscr{\widehat P}_q(\widehat B_q)$ is bounded and bounded away from zero even when $q$ approaches the singular set.

\subsection{Popp's measure on the singular set}

What happens for the construction of Popp's measure near a singular point $p$? The element $\varpi_q=\theta_q^{-1}(\mathscr{V}_q)$ in $\Lambda^n T^*_qM $ is still defined canonically (up to a sign) at any point $q$ near $p$ (at least when $\bD$ is a distribution). However the function $q \mapsto \varpi_q$ is not smooth at $p$ since the graded vector space $\mathfrak{gr}_q(\bD)$ does not vary continuously near the singular point, thus it is not possible to define a volume form $\varpi$.  So Popp's measure is a well defined (canonical) smooth volume   only on equiregular manifolds.

To our knowledge, a notion in the non equiregular case has never been proposed in literature. We provide here a way of defining Popp's measure under a standing assumption on the structure of the set of singular points. The idea is to generalize Popp's measure on particular singular sets, the so-called equisingular submanifolds and assume that the singular set is stratified by those.

Let us recall the notion of equisingular submanifold, see \cite{noi2,gjha}.
\begin{definition}
Let $N\subset M$ be a smooth connected submanifold and $q \in N$. The \emph{flag  at $q$ of $\,\bD$ restricted to $N$} is the sequence of subspaces
\begin{equation} \label{flagnq}
\{0\} \subset (\bD_q^1\cap T_qN)\subset\dots\subset(\bD_q^{r(q)}\cap T_qN)=T_qN.
\end{equation}
Set
\begin{eqnarray}
& n_i^N(q)= \dim (\bD_q^i\cap T_qN)\qquad \hbox{and} \qquad
Q_N(q)=\sum_{i=1}^{r(q)}i(n^N_{i}(q)-n^N_{i-1}(q)).& \label{defqn}
\end{eqnarray}
 We say that
$N$ is {\it equisingular} if, for every $i$, both dimensions $n_i(q)=\dim \bD_q^i$ and $n_i^N(q)$ are constant as $q$ varies in $N$.
In this case, we denote by $Q_N$  the constant value of $Q_N(q)$.
\end{definition}

Note that an equisingular submanifold is open if and only if it is contained in the regular set. Otherwise, it is contained in the singular set, which motivates the name.
\begin{example}[Grushin plane]\label{gru} Let $M=\R^2$, $\bD=\span\{X_1,X_2\}$ where $X_1=\partial_{x_1}$ and $X_2=x_1\partial_{x_2}$, and $g$ be the metric obtained by declaring $\{X_1,X_2\}$ to be orthonormal.
For every $x=(x_1,x_2)$ with $x_1\neq 0$,  $n_1(x)=2$ and thus $\Reg=\{(x_1,x_2)\in\R^2\mid x_1\neq 0\}$. The set of singular points coincides with the vertical axis $N=\{(0,x_2)\mid x_2\in\R\}$   and for every $x\in N$ $n_1(x)=1, n_2(x)=2$, $n^N_1(x)=0$, $n_2^N(x)=1$. Hence, $N$ is equisingular and $Q_N=2$.
\end{example}
\begin{example}[Martinet space]\label{mar}
Let $M=\R^3$, $\bD=\span\{X_1,X_2\}$, where
\begin{equation*}
X_1=\partial_{x_1},\quad X_2=\partial_{x_2}+\frac{x_1^2}{2}\partial_{x_3},
\end{equation*}
and $g$ be the metric obtained by declaring $\{X_1,X_2\}$ to be orthonormal.
For every $x\in\R^3$ with $x_1\neq 0$, $n_1(x)=2, n_2(x)=3$ and $\Reg=\{x\in\R^3\mid x_1\neq 0\}$. The set of singular points is the plane $N=\{(0,x_2,x_3)\mid (x_2,x_3)\in\R^2\}$ and, for every $x\in N$, $n_1(x)=n_2(x)=2, n_3(x)=3$, whereas $n^N_1(x)=n^N_2(x)=1, n^N_3(x)=2$. Thus  $N$ is equisingular and $Q_N=4$.
\end{example}

In \cite{gjha} we  studied properties of the spherical measure\footnote{For every $\alpha\geq 0$, we denote by $\ss^\alpha_N$ the $\alpha$-dimensional spherical Hausdorff measure in the metric space $(N, d|_N)$, where $d$ is the sub-Riemannian distance on $M$.} $S^{Q_N}_N$ exploiting the algebraic structure associated with an equisingular submanifold $N$ as provided below.
\begin{lemma}[see {\cite[Proposition~5.1]{gjha}}]\label{5.1}
Let $N\subset M$ be an  equisingular submanifold.  Then $\dim_H N=Q_N$ and, for every $p\in N$,
the graded vector space
\begin{equation*}
 \mathfrak{gr}_p^N(\bD):= \bigoplus_{i=1}^{r(p)}(\bD^{i}_p\cap T_pN)/(\bD^{i-1}_p\cap T_pN)
 \end{equation*}
 is a nilpotent Lie algebra whose associated Lie group $\mathrm{Gr}_p^N(\bD)$ is diffeomorphic to $T_pN$. Moreover $\mathfrak{gr}_p^N(\bD)$ varies smoothly w.r.t.\ $p$ on $N$, i.e.\ $\mathfrak{gr}^N(\bD)$ defines a vector bundle on $N$.
\end{lemma}
Note that the Lie group $\mathrm{Gr}_p^N(\bD)$ is not a Carnot group in general since the graded Lie algebra $\mathfrak{gr}_p^N(\bD)$ may not be generated by its first homogeneous component.
\medskip

Let $N$ be an oriented equisingular submanifold of dimension $k$. Thanks to Lemma~\ref{5.1}, the construction of Popp's measure done previously on a connected component $U\subset\Reg$ can be extended on $N$ in the following way. Assume first that $\bD$ is a distribution on $M$.

\begin{enumerate}[(i)]
\iii As in the regular case, use the maps $\pi_i:\otimes^i\bD_p\to\bD_p^i/\bD^{i-1}_p$ given by \eqref{eq:pii} to define an inner product on $\mathfrak{gr}_p(\bD)$ for every $p\in N$.
\iii Since $\mathfrak{gr}_p^N(\bD)$ can be canonically identified with a linear subspace of $\mathfrak{gr}_p(\bD)$, we obtain by restriction an inner product on $\mathfrak{gr}_p^N(\bD)$ and thus canonically (up to a sign) an element $\mathscr{V}_p^N\in \Lambda^k\mathfrak{gr}_p^N(\bD)^*$ for every $p\in N$. Because $\mathfrak{gr}^N(\bD)$ is a vector bundle on $N$, $\mathscr{V}_p^N$ depends smoothly on $p \in N$.
\iii Apply Lemma~\ref{alg} b) to $E=T_pN$ and $F=\bD_p \cap T_pN$ to get a canonical isomorphism $\theta_p^N: \Lambda^kT^*_pN \to \Lambda^k\mathfrak{gr}^N_p(\bD)^*$.
\iii Set $\mathscr{P}^N$ as the measure associated with the smooth volume form $p \mapsto (\theta_p^N)^{-1}(\mathscr{V}^N_p)$ on $N$.
\end{enumerate}

When $\bD$ is not a distribution, that is when the sub-Riemannian manifold is defined by a sub-Riemannian structure $(\mathbf{U},\ps,f)$ on $M$ (see Definition~\ref{SR}), the maps $\pi_i$ are not well-defined. In this case we have to replace every $\pi_i$, $i=1,\dots, r(p)$, in point (i) by the map $\widetilde{\pi}_i:\otimes^i\mathbf{U}_p \to \bD_p^i/\bD^{i-1}_p$ given by
\begin{equation*}
\widetilde{\pi}_i(v_1\otimes\dots\otimes v_i)=[f(V_1),\dots,[f(V_{i-1}), f(V_i)]]](p) \mod \bD^{i-1}_p,
\end{equation*}
where $V_1,\dots, V_i$ are smooth sections of $\mathbf{U}$ satisfying $V_j(p)=v_j$. The rest of the construction is unchanged.

\begin{corollary}
Let $N$ be an oriented equisingular submanifold. Then $\mathscr{P}^N$ is a   smooth volume   on $N$.
\end{corollary}
Note that when $N$ is an open connected subset $U$ of the regular set, the above construction  of $\mathscr{P}^N$ coincides with the one of $\mathscr{P}^U$ given at the beginning of this section. \medskip

We have now all the ingredients to define a global measure on $M$.
We make the assumption that the singular set $\Sing = M \setminus \Reg$  is \emph{stratified by equisingular submanifolds}, that is,   $\Sing$ is a countable union $\cup_i\Sing_i$ of disjoint equisingular submanifolds $\Sing_i$. Note that the regular set $\Reg$ always admits such a stratification as the union of its connected components $\Reg=\cup_i \Reg_i$, every $\Reg_i$ being an open equisingular submanifold, so that the whole $M$ is stratified by equisingular submanifolds. Since for every $i$, $\dim_H \Sing_i= Q_{\Sing_i}$ and $\dim_H \Reg_i = Q_{\Reg_i}$ (which is the constant value of $Q(q)$ on $\Reg_i$), there holds
\begin{equation*}
\dim_H \Reg = \sup_i Q_{\Reg_i}, \quad \dim_H \Sing = \sup_i Q_{\Sing_i}, \quad \dim_H M = \max \{\dim_H \Reg, \dim_H \Sing \}.
\end{equation*}
Note that in a non compact manifold $M$ it may happen that the integers $Q_{\Sing_i}$ are unbounded, and so that $\dim_H M = \infty$. We assume here that it is not the case.

\begin{definition}
Let $(M,\bD, g)$ be an oriented sub-Riemannian manifold stratified by equisingular submanifolds $M=\cup_i N_i$. We define the measures $\mathscr{P}_1$ and $\mathscr{P}_2$ on $M$ by
  \begin{eqnarray*}
 && \mathscr{P}_1 \llcorner_{N_i}=
  \begin{cases}\mathscr{P}^{N_i}, &\textrm{ if }Q_{N_i}=\dim_H M,\\
  0, &\textrm{ otherwise};
\end{cases}\\[2mm]
  &&\mathscr{P}_2\llcorner_\Reg = \mathscr{P}^{\Reg}, \mathscr{P}_2\llcorner_\Sing = 0.
  \end{eqnarray*}
In other terms, with a little abuse\footnote{we keep the notation $\mathscr{P}^{N_i}$ for the measure obtained by extending $\mathscr{P}^{N_i}$ so that it vanishes identically outside $N_i$. } of notations,
\begin{equation*}
\mathscr{P}_1 = \sum_{i \in {\mathcal J}}\mathscr{P}^{N_i} \qquad \hbox{and} \qquad \mathscr{P}_2 = \mathscr{P}^{\Reg}=\sum_{i}\mathscr{P}^{\Reg_i},
\end{equation*}
where ${\mathcal J}=\{i \mid Q_{N_i}=\dim_H M\}$.
\end{definition}

The first measure $\mathscr{P}_1$ has two main properties: it takes account of the singular set and it is absolutely continuous with respect to the Hausdorff measure $\ss^{\dim_HM}$ (indeed every $\mathscr{P}^{N_i}$ is a smooth volume on $N_i$ and so absolutely continuous w.r.t.\ $\ss^{Q_{N_i}}$ by \cite[Theorem~5.3]{gjha}). More precisely it charges all the submanifolds $N_i$ having the same Hausdorff dimension as $M$. For instance,
in Example~\ref{gru} the singular set $\Sing$ is equisingular and has Hausdorff dimension $2$, as the regular set, so that $\mathscr{P}_1=\mathscr{P}^\Sing+ \mathscr{P}^{\Reg}$. The same holds for the Martinet space, see Example~\ref{mar}.

The second measure $\mathscr{P}_2$ is obtained by simply charging zero mass to the singular set. Under the stratification assumption, the singular set is $\mu$-negligible  for every smooth volume  $\mu$ on $M$ (recall that the singular set always has an empty interior and so no open equisingular strata). As a consequence, thanks to Proposition~\ref{weq}, $\mathscr{P}_2$ is absolutely continuous with respect to any smooth volume   $\mu$ on $M$.

However neither $\mathscr{P}_1$ nor $\mathscr{P}_2$ have a $L^\infty$ density with respect to a smooth volume $\mu$.
\medskip

Finally, note that on every equisingular submanifold $N$ we have a weak estimate of $\mathscr{P}^N$ similar to the one of Proposition~\ref{weq}.
\begin{proposition}\label{weqN}
Let $M$ be an $n$-dimensional oriented manifold with a volume form $\omega$, and
 $N \subset M$ be a $k$-dimensional oriented equisingular submanifold with a volume form $\bar \omega$. We denote by $\bar \mu$ the associated smooth volume on $N$. Finally, let $X_1,\dots, X_{m}$ be a generating family of the sub-Riemannian structure $(\bD, g)$ on $M$.

For any compact subset $K \subset M$ there exists a constant
$C>0$ such that, for every $q \in N \cap K$,
\begin{equation*}
 \frac{C^{-1}}{\bar \nu(q)} \leq \frac{d\mathscr{P}^{N}}{d\bar \mu}(q) \leq  \frac{C}{\bar \nu(q)},
\end{equation*}
where ${\bar \nu(q)} = \max\{ \left(\bar \omega\wedge dY_{k+1} \wedge \cdots
  \wedge dY_n\right)_q (Y_1(q),\dots,Y_n(q)) \}$, the
maximum being taken among all $n$-tuples $(Y_1,\dots, Y_n)$ which are frames of brackets adapted at $q$ maximizing $\omega_q(Y_1(q),\dots,Y_n(q))$ (i.e.\ so that $\omega_q(Y_1(q),\dots,Y_n(q))=\nu(q)$).
\end{proposition}

Using \cite[Remark 5.9]{gjha}, we obtain as a corollary that  $\mathscr{P}^N$ is commensurable with the Hausdorff measure $\ss^{Q_N}\llcorner_N$ on $N$.
\begin{corollary}
For any compact subset $K \subset M$ there exists a constant
$C>0$ such that, for every $q \in N \cap K$,
\begin{equation*}
\frac{1}{C}\leq \frac{d\mathscr{P}^N}{d\ss^{Q_N}\llcorner_N}(q) \leq C.
\end{equation*}
\end{corollary}


\end{document}